\pdfoutput=1
\documentclass{amsart}
\usepackage{longtable}

\usepackage{mathrsfs}

\usepackage{hyperref}
\hypersetup{
    colorlinks=true,       
    linkcolor=blue,          
    citecolor=blue,        
    filecolor=blue,      
    urlcolor=blue           
}
\usepackage[alphabetic,backrefs,lite]{amsrefs}

\usepackage[alphabetic,backrefs,lite]{amsrefs}
\usepackage{todonotes}
\usepackage{booktabs}
\usepackage{pdflscape}

\usepackage[all]{xy}
\usepackage[utf8]{inputenc}
\usepackage{pgf,tikz}
\usetikzlibrary{arrows}
\usetikzlibrary{matrix,decorations.pathreplacing}

\numberwithin{equation}{section}

\theoremstyle{plain}
\newtheorem{lemma}{Lemma}[section]
\newtheorem{theorem}[lemma]{Theorem}
\newtheorem*{introthm}{Theorem}
\newtheorem{proposition}[lemma]{Proposition}

\theoremstyle{definition}
\newtheorem{definition}[lemma]{Definition}

\newtheorem{construction}[lemma]{Construction}

\theoremstyle{remark}
\newtheorem{remark}[lemma]{Remark}

\newcommand{\cc}{\mathbb{C}}

\newcommand{\pp}{\mathbb{P}}
\newcommand{\ff}{\mathbb{F}}
\newcommand{\qq}{\mathbb{Q}}

\newcommand{\zz}{\mathbb{Z}}

\newcommand{\Osh}{{\mathcal O}}

\newcommand{\R}{{\mathcal R}}

\newcommand{\Cone}{\operatorname{Cone}}

\newcommand{\Pic}{\operatorname{Pic}}
\newcommand{\Cl}{\operatorname{Cl}}

\newcommand{\Nef}{\operatorname{Nef}}
\newcommand{\Mov}{\operatorname{Mov}}

\newcommand{\NE}{\operatorname{NE}}

\newcommand{\mw}{\operatorname{MW}}

\newcommand{\id}{\operatorname{id}}

\begin{document}

\title{On cubic elliptic varieties}

\author[J.~Hausen]{J\"urgen Hausen}
\address{Mathematisches Institut, Universit\"at T\"ubingen,
Auf der Morgenstelle 10, 72076 T\"ubingen, Germany}
\email{juergen.hausen@uni-tuebingen.de}

\author[A.~Laface]{Antonio Laface}
\address{
Departamento de Matem\'atica,
Universidad de Concepci\'on,
Casilla 160-C,
Concepci\'on, Chile}
\email{alaface@udec.cl}

\author[A.~L.~Tironi]{Andrea Luigi Tironi}
\address{
Departamento de Matem\'atica,
Universidad de Concepci\'on,
Casilla 160-C,
Concepci\'on, Chile}
\email{atironi@udec.cl}

\author[L.~Ugaglia]{Luca Ugaglia}
\address{
Dipartimento di Matematica e Informatica,
Universit\`a degli studi di Palermo,
Via Archirafi 34,
90123 Palermo, Italy}
\email{luca.ugaglia@unipa.it}


\begin{abstract}
Let $\pi:X\to \pp^{n-1}$ be an elliptic fibration 
obtained by resolving the indeterminacy of the 
projection of a cubic hypersurface $Y$ of $\pp^{n+1}$
from a line $L$ not contained in $Y$. We prove
that the Mordell-Weil group of $\pi$ is finite
if and only if the Cox ring of $X$ is finitely generated.
We also provide a presentation of the Cox ring of $X$
when it is finitely generated.
\end{abstract}
\maketitle

\section*{Introduction}
Let $\pi: X\to Z$ be an elliptic fibration between
smooth projective complex varieties which admits
a section. The generic fiber $X_\eta$ of $\pi$
is an elliptic curve over the function field
of $Z$ and its group of rational points (the
Mordell-Weil group of $\pi$) reflects
into the geometry of $X$. It is thus interesting
to explore the relation between the Mordell-Weil
group of $\pi$ and the Cox ring of the 
variety $X$. In this paper we 
focus on a class of elliptic fibrations defined
by the linear system
$|-\frac{1}{n-1}K_X|$, where $X$ is a certain
blowing-up of a smooth cubic $n$-dimensional
hypersurface.
Inspired by the recent work~\cite{CP} we
classify such fibrations according to their
Mordell-Weil groups and prove the following.
\begin{introthm}\label{moridream}
Let $X$ be an elliptic variety, of dimension at 
least three, obtained
by resolving the indeterminacy of the
projection of a smooth cubic hypersurface from a line.
Then the following are equivalent:
\begin{enumerate}
\item the Cox ring of $X$ is finitely generated;
\item the Mordell-Weil group of the elliptic fibration
is finite.
\end{enumerate}
Moreover when (1) and (2) hold we provide an explicit
presentation for the Cox ring in Theorem~\ref{teor}.
\end{introthm}

Observe that if $X$ is as in the statement of the
preceding theorem, $D$ is a general divisor in
the linear system $|-mK_X|$, for $m>1$ and 
$\Delta=\frac{1}{m}D$, then $(X,\Delta)$ 
is a {\em klt Calabi-Yau pair} (see~\cite{CP}). 
As a byproduct of the theorem we obtain that
if the fibration $X$ has finite Mordell-Weil group,
then the Morrison-Kawamata cone conjecture
for klt Calabi-Yau pairs holds.

The paper is structured as follows. In Section~1  
we prove some facts about elliptic fibrations
$\pi: X\to\pp^{n-1}$ with a section and such that 
$-K_X$ is a multiple of the preimage of a general 
hyperplane of $\pp^{n-1}$. In Section~2 we introduce
a particular case of elliptic fibration, i.e.
the blow up of a cubic hypersurface $Y$ 
of $\pp^{n+1}$ along the intersection $Y\cap L$
with a line not contained in $Y$ and we state
some general results about it.
In Section~3 we study the nef and
moving cones of these varieties and 
we finally prove that, for these cubic elliptic varieties
the finite generation of the Cox ring is equivalent
to the finiteness of the Mordell-Weil group of $\pi$.
Finally in Section~4 we give a presentation for the
Cox ring of $X$ when it is finitely generated.

\section{Elliptic fibrations}

Let $X$ be a smooth projective variety
and let $\pi: X\to\pp^{n-1}$ be an {\em elliptic
fibration}, that is a general fiber of $\pi$
is a smooth irreducible curve of genus one.
\begin{definition}\label{co1}
The fibration $\pi$ is {\em jacobian}
if it admits a section. If this is the case
the  {\em Mordell-Weil
group} of $\pi$ is the group of sections
$\sigma:\pp^{n-1}\to X$, that is
\[
 \mw(\pi)
 :=
 \{\sigma:\pp^{n-1}\to X : \pi\circ\sigma=\id\}.
\]
We say that the fibration $\pi:X\to\pp^{n-1}$
is {\em extremal} if its Mordell-Weil group
is finite.
Moreover we say that
$\pi$ is {\em relatively minimal} if, for a general
line $R$ of $\pp^{n-1}$ the restriction of $\pi$
to the elliptic surface $S=\pi^{-1}(R)$ does
not contract $(-1)$-curves.

\end{definition}

Observe that by the Riemann-Roch theorem
the set of sections of $\pi$
is in bijection with the group $\Pic^0(X_\eta)$,
where $X_\eta$ is the generic fiber of
$\pi$. 

\begin{proposition}
Let $\pi: X\to\pp^{n-1}$ be a jacobian elliptic fibration
and assume that $K_X$ is linearly equivalent to 
$\alpha\,\pi^*\Osh_{\pp^{n-1}}(1)$, where $\alpha$
is a rational number.
Then $\pi$ is relatively minimal and $\alpha$
is integer.
Moreover, if $S$ is the preimage of a general
line of $\pp^{n-1}$, then the following are equivalent:
\begin{enumerate}
\item $S$ is a rational surface;
\item $\alpha = 1-n$.
\end{enumerate}
\end{proposition}
\begin{proof}
Consider a general flag of linear subspaces of $\pp^{n-1}$. The
corresponding preimages via $\pi$ give a flag of subvarieties
$F_i$ of $X$
\[
 X\supset F_{n-1}\supset\dots\supset F_2=S\supset F_1=f,
\]
where $\dim F_i=i$ for any $i$
and $f$ is a fiber of $\pi$. By hypothesis
and the adjunction formula we get $K_S\sim (n-2+\alpha)f$. If
$C$ is a $(-1)$-curve of $S$, then $-1=C\cdot K_S=C\cdot
(n-2+\alpha) f$ implies that $C$ can not be contained in a fiber
of $\pi$, so that $\pi$ is relatively minimal. Moreover, observe
that given a section $\sigma$, the curve 
$\Gamma=\sigma(\pp^{n-1})\cap S$ is a section of
$\pi|_S$, so that $\Gamma\cdot f=1$. Hence 
$n-2+\alpha=\Gamma\cdot K_S$ is integer
so that $\alpha$ is integer too.

$(1)\Rightarrow (2)$.
Since $S$ is a rational surface and
$K_S\sim (n-2+\alpha) f$, then
$\alpha\leq 1-n$ and in particular
$\Gamma\cdot K_S<0$. 
Observe that the divisor $K_S-\Gamma$
can not be linearly equivalent to an 
effective divisor since $(K_S-\Gamma)\cdot f=-1$.
Hence $h^2(S,\Gamma)=0$,
by Serre's duality. Moreover, since $\Gamma$ is a
section of $\pi|_S$ we have $h^0(S,\Gamma)=1$.
Hence by Riemann-Roch
\[
 1=h^0(S,\Gamma)
 \geq
 \chi(S,\Gamma)
 =
 \frac{\Gamma^2-\Gamma\cdot K_S}{2}+1
\]
which implies $\Gamma^2\leq\Gamma\cdot K_S<0$. Thus $\Gamma$ is a
$(-1)$-curve and in particular $n-2+\alpha=\Gamma\cdot K_S=-1$
giving $\alpha=1-n$.

$(2)\Rightarrow (1)$.
Since $\alpha=1-n$ then $K_S\sim -f$, so that
$S$ has negative Kodaira dimension.
By the classification theory of surfaces $S$ 
is either rational or the blowing-up of a ruled surface. 
Since $K_S^2=0$, by~\cite[\S5,
Cor. 2.11]{Ha} we conclude that $S$ is rational.
\end{proof}

\begin{proposition}\label{sample}
Let $\pi: X\to\pp^{n-1}$ be a jacobian elliptic fibration
and assume that $K_X$ is linearly equivalent to a negative
multiple of the pull-back of $\Osh_{\pp^{n-1}}(1)$.
Then any nef effective divisor of $X$ is semiample.
\end{proposition}
\begin{proof}
Let $D$ be a nef effective divisor of $X$. 
Since both $D$ and $-K_X$ are nef, then
$D-K_X$ is nef. If $D-K_X$ is also big, then
$D$ is semiample by the Kawamata--Shokurov 
base point free theorem
(see~\cite{KMM} and~\cite{Shokurov}). 
If $D-K_X$ is not big, then $(D-K_X)^n=0$ 
and in particular $D\cdot(-K_X^{n-1})=0$. 
By hypothesis $-K_X^{n-1}$ is rationally
equivalent to a positive multiple of a fiber of $\pi$.
Hence, since $D$ is effective, its support 
is the preimage of a hypersurface of $\pp^{n-1}$.

We conclude by showing that $D$ 
is the pull-back of a divisor of $\pp^{n-1}$,
so that it is semiample.
Indeed if this is not the case, let $S$ be the preimage
of a general line of $\pp^{n-1}$. Then 
$(D|_S)^2<0$, by~\cite[Lemma 8.2]{BPV},
a contradiction since $D$ is nef.
\end{proof}

\section{Generalities on cubic elliptic varieties}
From now on we will concentrate on the case in which
$X$ is obtained from a cubic hypersurface $Y$
of $\pp^{n+1}$
by resolving the indeterminacy locus of the projection map
from a line $L$ non contained in $Y$.
Therefore the variety $X$
comes with two morphisms:
\[
 \xymatrix@C=5pt{
  X\ar[r]^-\pi\ar[d]_-\sigma & \pp^{n-1}\\
  \hspace{1.3cm}Y\subset\pp^{n+1}
 }
\]
where $\pi$ is the elliptic fibration
while $\sigma$ is the resolution of
the indeterminacy.
Observe that the fibers
of $\pi$ are the strict transforms
of the plane cubics cut out on $Y$
by planes containing $L$.

\begin{remark}\label{bu}
The birational morphism $\sigma$
is composition of three blowing-ups

\[
 \xymatrix{
  X\ar[r]^-{\sigma_3}\ar@/_{10pt}/[rrr]_\sigma &
  Y_2\ar[r]^-{\sigma_2} &
  Y_1\ar[r]^-{\sigma_1} &
  Y
 }
\]
at the points $p_1,p_2,p_3$.
There are three possibilities (modulo a relabelling
of the three points):
\begin{enumerate}
\item the points $p_2$ and $p_3$ do not lie
on the exceptional divisors;
\item $p_2$ lies on the exceptional
divisor of $\sigma_1$;
\item $p_2$ lies on the exceptional
divisor of $\sigma_1$ and $p_3$
on that of $\sigma_2$.
\end{enumerate}
\end{remark}

In what follows we denote by
$H$ the pull-back of a hyperplane
of $Y$ and by $E_i$ the exceptional
divisor of $\sigma_i$, for $i\in\{1,2,3\}$.
In case (1) each $E_i$ is isomorphic
to $\pp^{n-1}$. In case (2) the prime divisor
$E_1-E_2$ is isomorphic to 
the projectivization $\mathbb{F}$ of the vector bundle
$\Osh_{\pp^{n-1}}\oplus\Osh_{\pp^{n-1}}(1)$ while 
$E_2$ and $E_3$ are both isomorphic
to $\pp^{n-1}$. 
Finally in case (3) the prime divisors
$E_1-E_2$ and $E_2-E_3$ are isomorphic to 
$\ff$ while $E_3$ is isomorphic to $\pp^{n-1}$.
In each case 
\[
 \Pic(X) = \langle H, E_1, E_2, E_3\rangle,
\]
where, with abuse of notation, we are adopting
the same symbols for the divisors and for their
classes.

\subsection{Cubic elliptic varieties}
Let us recall the following definition
(see~\cite{CC}):
\begin{definition}
Given a hypersurface $Y$ of $\pp^{n+1}$ of degree $d$,
a smooth point $p$ of $Y$ is said to
be a {\em star point} if $T_pY\cap Y$ has multiplicity $d$ at $p$.
\end{definition}

Let us consider now the local
study of a cubic $Y$ at
a smooth point $p$. After applying
a linear change of coordinates we
can assume $p=(0:\ldots:0:1)$ and the
equation of the tangent space to $Y$
at $p$ to be $T_{n+1}=0$. Hence a defining
equation for $Y$ is
\begin{equation}\label{eq-cubic}
 T_{n+1}\,a+T_{n+2}\,b(T_1,\ldots,T_n)+c(T_1,\ldots,T_n) = 0,
\end{equation}
where $a$ is a degree two homogeneous
polynomial while $b$ and $c$ are
homogeneous polynomials of $\cc[T_1,\ldots,T_n]$
of degrees two and three respectively.
Observe that $c$ can not be the zero
polynomial, since otherwise $Y$ would
contain the linear space $V(T_{n+1},T_{n+2})$ being
singular.

Observe that any line $R$ of $Y$ through $p$ is
contained in the tangent space
$T_pY$ and so it is contained in the
intersection of the two cones
$V(b)\cap V(c)$.
\begin{proposition}\label{flexes}
Let $Y$ be a smooth cubic hypersurface
of $\pp^{n+1}$, let $p$ be a point of $Y$.
Assume that a local equation of $Y$ at $p$
is~\eqref{eq-cubic}.
Then the following properties hold:
\begin{enumerate}
\item $p$ is a star point of $Y$
if and only if $b$ is the zero polynomial;
\item if $p$ is not a star point then there 
is a $(n-3)$-dimensional family of lines 
of $Y$ passing through it;
\item a line through two star points
of $Y$, intersects $Y$ at a third star point.
\end{enumerate}
\end{proposition}
\begin{proof}
Point $(2)$ is an immediate consequence of
our previous discussion, while $(1)$ follows
by observing that a general line tangent to 
$Y$ at $p$ has parametric equation
\[
 u\,(0,\dots,0,1) + v\,(t_1,\dots,t_n,0,t_{n+2}),
\]
where the $t_i$ are general complex
numbers.
By substituting in~\eqref{eq-cubic}
it follows that the left hand side is
a cubic polynomial in $u$ and $v$
and it has a root of multiplicity three
for any choice of the $t_i$ 
if and only if $b$ vanishes identically.

To prove $(3)$
first consider the case when $L$
is tangent to $Y$ at the star point
$p_1$. Then $L$ intersects $Y$ at $p_1$
with multiplicity three by definition o star point.
Assume now that $L$ intersects $Y$
at three distinct points $p_1,p_2,p_3$
such that $p_1$ and $p_2$ are star points.
After a linear change of coordinates
we can assume $p_1=(0:\ldots:0:1)$
with $T_{p_1}Y$ of equation $T_{n+1}=0$
and $p_2=(0:\ldots:0:1:0)$ with 
$T_{p_2}Y$ of equation $T_{n+2}=0$.
Using equation~\eqref{eq-cubic} and
(1) we get that a defining equation for $Y$ is
\[
 T_{n+1}T_{n+2}\,\ell + c(T_1,\dots,T_n)
 =
 0,
\]
where $\ell$ is a linear form. Hence 
$p_3=(0:\dots:0:\alpha:\beta)$, where
$\ell(p_3)=0$. The fact that $p_3$ is
a star point follows immediately from
the previous equation for $Y$, being
$\ell=0$ the equation of $T_{p_3}Y$.
\end{proof}

As a consequence of this result and
of Remark~\ref{bu} we have
that there are seven different possibilities
concerning the points $L\cap Y$. 
We are now going to construct a table 
in which we list the seven types of 
cubic elliptic varieties
we can obtain.
In the first column we write the type
of the variety using a symbol 
that recalls which points we
are blowing up and in which order.
For example if $X$ is a blowing-up
at three distinct non-star points, then we will denote it by
$X_{111}$, while if $X$ is blowing-up of one star point and two
non-star infinitely near points we will denote it by $X_{S2}$.

The second column contains the defining
equations of $Y$ and the line $L$ while the third column
is for the Mordell-Weil groups of the elliptic fibrations.

\begin{center}
\footnotesize
\begin{longtable}{clc}
\hline
Type & Defining equations for $Y$ and $L$ & Mordell-Weil group
\\
\hline
\\
$X_{3}$ &
\begin{tabular}{l}
$T_{n+1}(a'+T_{n+2}\,a_1)+T_{n+2}\,b'+b_1=0$\\
$T_1=\dots=T_{n-1}=T_{n+1}=0$
\end{tabular}
& 
$\langle 0\rangle$
\\
\hline
\\
$X_{S}$ & 
\begin{tabular}{l}
$T_{n+1}\,a_2+b_2=0$\\
$T_1=\dots=T_{n-1}=T_{n+1}=0$
\end{tabular}
& $\langle 0\rangle$
\\
\hline
\\
$X_{S2}$ &
\begin{tabular}{l}
$T_{n+1}\,a_3+b_3=0$\\
$T_1=\dots=T_{n}=0$
\end{tabular}
& $\zz/2\zz$
\\
\hline
\\
$X_{SSS}$ & 
\begin{tabular}{l}
$T_{n+1}\,T_{n+2}\,a_4+b_4=0$\\
$T_1=\dots=T_{n}=0$
\end{tabular}
& $\zz/3\zz$
\\
\hline
\\
$X_{12}$ & 
\begin{tabular}{l}
$T_{n+1}\,a_5+T_{n+2}\,b_5+c_5=0$\\
$T_1=\dots=T_{n-1}=T_{n+1}=0$
\end{tabular}
& $\zz$
\\
\hline
\\
$X_{S11}$ & 
\begin{tabular}{l}
$T_{n+1}\,a_6+b_6=0$\\
$T_1=\dots=T_{n}=0$
\end{tabular}
& $\zz$
\\
\hline
\\
$X_{111}$ & 
\begin{tabular}{l}
$T_{n+1}\,a_7+T_{n+2}\,b_7+c_7=0$\\
$T_1=\dots=T_{n}=0$
\end{tabular}
& $\zz\oplus\zz$
\\
\hline
\\
\caption{The seven types of cubic elliptic varieties}
\label{types}
\end{longtable}
\end{center}

Where $b', b_i,c_i\in\cc[T_1,\dots,T_n]$, 
$a'\in\cc[T_1,\dots,T_{n+1}]$,
$a_i\in\cc[T_1,\dots,T_{n+2}]$,
moreover $b'$ does not contain $T_n^2$
and $a_3$  does not contain
$T_{n+1}^2$ and $T_{n+1}\,T_{n+2}$.
The equations appearing in the table
can be obtained from~\eqref{eq-cubic}
with a case by case study of the tangency
conditions at the points of $L\cap Y$ (as we did
in the proof of Proposition~\ref{flexes} for
$X_{SSS}$).

\subsection{Mordell-Weil groups}
Recall that the Mordell-Weil group
of the elliptic fibration $\pi$ is the
group of sections of $\pi$ or equivalently
the group of $K=\cc(\pp^{n-1})$-rational
points $X_\eta(K)$ of the generic fiber $X_\eta$ of 
$\pi$ once we choose one of such points 
$O$ as an origin for the group law.
Let $\mathscr{T}$ be the subgroup of
$\Pic(X)$ generated by the classes of
{\em vertical divisors}, that is
divisors mapped to hypersurfaces by $\pi$,
and by the class of the section
$O$. There is an exact
sequence~\cite[Sec. 3.3]{W}:
\begin{equation}\label{mw}
 \xymatrix{
  0\ar[r] &
  \mathscr{T}\ar[r] &
  \Pic(X)\ar[r] &
  X_\eta(K)\ar[r] &
  0.
 }
\end{equation}

\begin{theorem}
The Mordell-Weil group of the elliptic
fibration for each type in Table~\ref{types}
is the one given in the third column.
\end{theorem}
\begin{proof}
Let $X$ be one of the cubic elliptic varieties
appearing on the first column of
Table~\ref{types}. As already observed
in Section 2, the Picard group of
$X$ is free of rank four and is generated
by the classes of $H, E_1, E_2, E_3$.
Observe that since $p_3$
is the last point that we blow up
then $E_3$ gives a section of the
elliptic fibration $\pi$ so that from
now on we take $O=E_3$. 
The subgroup $\mathscr{T}$
has rank at least two, since it
contains the subgroup
\[
 \mathscr{L}
 =
 \langle H-E_1-E_2-E_3, E_3\rangle,
\]
where the first class is that of the pull-back of 
$\Osh_{\pp^{n-1}}(1)$.
Hence by~\eqref{mw} the Mordell-Weil 
group of $X$ has rank at most two. 

Consider now a prime vertical
divisor $D$ of $\pi$. By identifying
$D$ with its support we have that
$\pi(D)$ is a hypersurface $B$ of $\pp^{n-1}$.
If $D$ equals the pull-back $\pi^*B$
then it is linearly equivalent to a 
multiple of $H-E_1-E_2-E_3$.
If not, then any fiber $\Gamma=\pi^{-1}(q)$
over a point $q$ of $B$ is
reducible and has a component
contained in $D$. 
There are two possibilities for
the curve $C=\sigma(\Gamma)$, where
$\sigma: X\to Y$ is the blowing-up map:
\begin{enumerate}
\item $C$ is a reducible
cubic curve;
\item $C$ is an
irreducible singular cubic curve,
with singular point at
one of the points of $L\cap Y$.
\end{enumerate}
In the first case $C$ must contain a line,
so that one of the points $p$ of $L\cap Y$
is a star point and, named $E$ the 
corresponding exceptional divisor,
one of the irreducible components
of $\pi^*B$ is linearly equivalent to $H-3E$.
This shows that $\mathscr{T}=\mathscr{L}$
for $X_{111}$, that $\mathscr{T}=
\mathscr{L}+\langle H-3E_1\rangle$
for $X_{S11}$ and that $\mathscr{T}=
\mathscr{L}+\langle H-3E_1, H-3E_2\rangle$
for $X_{SSS}$.

In the second case $L$ is tangent to $Y$
at a point $p$ of $L\cap Y$.
Fibrations on the varieties $X_{12}$, $X_{3}$
and $X_{S}$ belong to this case.
We have $\mathscr{T}=
\mathscr{L}+\langle E_1-E_2\rangle$
for $X_{12}$ and $\mathscr{T}=
\mathscr{L}+\langle E_1-E_2, E_2-E_3\rangle$
for both $X_3$ and $X_{S}$.

Finally $X_{S2}$ belong to both cases
and we have $\mathscr{T}=
\mathscr{L}+\langle H-3E_1,E_2-E_3\rangle$.
We conclude by observing that the
Mordell-Weil group of each such elliptic
fibration is isomorphic to
$\Pic(X)/\mathscr{T}$.
\end{proof}

\subsection{A flop}\label{flop}
In this subsection we study a flop
image of the blowing-up $Y_1$ of a smooth
cubic hypersurface $Y$ of $\pp^{n-1}$
at a non-star point $p_1$.
The Picard group of $Y_1$ is free
of rank two generated by the classes
of the exceptional divisor $E$
and the pull-back $H$ of a hyperplane
section of $Y$.
Inside $\Pic(Y_1)\otimes_\zz\qq$
we have the following cones:\\

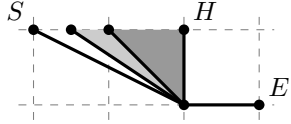
\begin{center}
\begin{tikzpicture}
\draw[step=1cm,gray,dashed]
(-2.2,-.2) grid (1.2,1.2);
\draw[line width=0mm,fill=gray!80!white] (0,1) -- (0,0) -- (-1,1);
\draw[line width=0mm,fill=gray!40!white] (-1,1) -- (0,0) -- (-1.5,1);
\foreach \x/\y in {0/1,1/0,-1/1,-1.5/1,-2/1}
 { \draw[-, very thick, black] (0,0) -- (\x,\y); }
\foreach \x/\y in {0/0,0/1,1/0,-1/1,-1.5/1,-2/1}
 { \fill[black] (\x,\y) circle (2pt); }
\node[above left] at (-2,1) {$S$};
\node[above right] at (1,0) {$E$};
\node[above right] at (0,1) {$H$};
\end{tikzpicture}
\end{center}
The cone generated by the 
classes of $H$ and $H-E$ is the
nef cone of $Y_1$, while the moving
cone is generated by the classes 
of $H$ and $H-\frac{3}{2}E$.
To prove this consider the birational
map
\begin{equation}\label{eq-flop}
 \psi: Y\to Y
 \qquad
 q\mapsto(\text{line}(p,q)\cap Y)-\{p,q\}.
\end{equation}
Denote by $\psi_1: Y_1\to Y_1$ the lift 
of $\psi$ to $Y_1$. Let $V$ be the strict
transform of the union of lines of $Y$ 
through $p$. Since $\psi_1$ is an 
involution whose indeterminacy locus
is $V$ and $V$ has codimension two
in $Y_1$, then $\psi_1$ is an isomorphism 
in codimension one. In particular it 
induces by pull-back an isomorphism
$\psi_1^*$ on the Picard group of $Y_1$.
To calculate the representative matrix
of $\psi_1^*$ with respect to the 
basis $(H,E)$, observe that $\psi$
maps points of the strict transform of
$T_pY\cap Y$ to points of the exceptional
divisor $E$ and viceversa.
The first divisor is linearly
equivalent to $H-2E$. Hence the
representative matrix for $\psi_1^*$
is:
\[
 \left(
 \begin{array}{rr}
  2  & 1\\
  -3 & -2
 \end{array}
 \right).
\]
The previous matrix explains the
$\zz/2\zz$-symmetry of the moving
and effective cones of $Y_1$.
If we blow up a set of points $Q$
on $Y_1$ then $\psi_1$ lifts to the blowing-up
if and only if $\psi_1(Q)=Q$.
This is exactly what happens for the
cubic elliptic varieties $X_{3}$ and $X_{S2}$.
In the first case each point is fixed by
$\psi_1$, while in the second case
the points $p_2$ and $p_3$ are exchanged.
This implies the following.
\begin{proposition}\label{floppic}
Let $X$ be a cubic elliptic variety of type
$X_3$ or $X_{S2}$ and let $\varphi:
X\to X$ be the flop induced by~\eqref{eq-flop}.
Then the action of
$\varphi^*$ on $\Pic(X)$
with respect to the basis $(H,E_1,E_2,E_3)$
is described respectively by the 
following two matrices
\[
M_3:=\left(
\begin{array}{rrrr}
 2 & 1 & 0 & 0\\
-3 & -2 & 0 & 0\\
 0 & 0 & 1 & 0\\
 0 & 0 & 0 & 1
\end{array}
\right)
\qquad
M_{S2}:=
\left(
\begin{array}{rrrr}
 2 & 1 & 0 & 0\\
-3 & -2 & 0 & 0\\
 0 & 0 & 0 & 1\\
 0 & 0 & 1 & 0
\end{array}
\right).
\]
\end{proposition}

\section{Nef and moving cones}
As a general reference about the 
cones discussed in this section
see~\cite{La}.

\begin{construction}\label{co-2}
In what follows we will write the classes
of $\Pic(X)$ with respect to the basis
$(H,E_1,E_2,E_3)$.
We fix the basis $(h,e_1,e_2,e_3)$ of
$A_1(X)$ such that the intersection pairing 
$\Pic(X)\times A_1(X)\to\zz$ in these coordinates
is given by
\[
 ((a,a_1,a_2,a_3),(b,b_1,b_2,b_3))
 \mapsto
 ab-a_1b_1-a_2b_2-a_3b_3.
\]
Observe that $h$ is the class of the
pull-back of a line of $Y$ and $e_3$
is the class of a line in the exceptional 
divisor $E_3$.
The geometric interpretation of the remaining elements
is the following. 
If we are blowing-up one point of $Y$
(cases $X_{3}$, $X_{S}$)
$e_2-e_3$  and $e_1-e_2$ are fibers
of the $\pp^1$-bundles
$E_2-E_3$ and $E_1-E_2$ respectively.
If we are blowing-up two points of $Y$
(cases $X_{12}$, $X_{S2}$)
$e_2$ is the class of a
line in the exceptional divisor $E_2$
while $e_1-e_2$ is a fiber of the $\pp^1$-bundle
$E_1-E_2$.
If we are blowing-up three
points of $Y$ (cases $X_{111}$, $X_{S11}$,
$X_{SSS}$) each $e_i$ is the class of a
line in the exceptional divisor $E_i$.
\end{construction}

\subsection{Nef cones} 
Let us compute now the nef cones of 
the cubic elliptic varieties of Table~\ref{types}.
In each case we will proceed as follows. 
We take some classes of nef divisors 
and we consider the cone $N$ they span.
Since the nef cone of $X$ is the dual of the 
Mori cone $\NE(X)$ of $X$ and $N$ is contained in the former, 
we deduce that the dual $N^*$ contains 
$\NE(X)$. We conclude by proving that the 
classes which generate $N^*$ are indeed classes
of effective curves and hence $\NE(X)=N^*$
so that $\Nef(X)=N$.

\begin{proposition}\label{nefX}
Let $X$ be one of the cubic elliptic varieties
of Table~\ref{types}. Then the
nef cone of $X$ is generated
by the semiample classes whose
coordinates with respect to the basis 
$(H,E_1,E_2,E_3)$ of $\Pic(X)$ 
are the columns of the corresponding 
matrix in the following table.\\

\begin{center}
\footnotesize
\begin{tabular}{ll}
\hline
{\rm Type} & {\rm Generators of the nef cone}
 \\
 \hline
 \\
 $X_3$, $X_{S}$ &
 $
 \left[
 \begin{array}{rrrr}
 1 & 1 & 1  & 1\\
 -1 & -1 & -1 & 0\\
 -1 & -1 & 0  & 0\\
 -1 & 0 & 0  & 0
 \end{array}
 \right]
 $
 \\
 \\
 \hline
 \\
 $X_{12}$, $X_{S2}$ &
 $
 \left[
 \begin{array}{rrrrrr}
  1 & 1 & 1 & 1 & 1 & 1\\
-1 &-1& -1& 0 &  -1 & 0\\
-1 & -1 & 0 & 0 & 0 & 0\\
-1 & 0 & 0 & 0 & -1 & -1
 \end{array}
 \right]
 $
 \\
 \\
 \hline
 \\
 $X_{111}$, $X_{S11}$, $X_{SSS}$ &
 $
 \left[
 \begin{array}{rrrrrrrr}
  1 & 1 & 1 & 1 & 1 & 1 & 1 & 1\\
-1 &-1& -1& 0 &  -1 & 0 & 0 & 0\\
-1 & -1 & 0 & 0 & 0 & 0 & -1 & -1\\
-1 & 0 & 0 & 0 & -1 & -1 & -1 & 0
 \end{array}
 \right]
 $
\end{tabular}
\end{center}

\end{proposition}
\begin{proof}
First of all observe that all the columns
of the previous matrices are degrees
of nef divisors (indeed semiamples) since
the class of $F=H-E_1-E_2-E_3$ is semiample
being the pull-back of $\Osh_{\pp^{n-1}}(1)$
and all the remaining columns are of
the form $\gamma^*\gamma_*F$ for 
some birational morphism $\gamma$
which is a composition of the contractions
$\sigma_i$.

We conclude by showing that the dual
of each cone generated by the columns of
the given matrices is contained in the
Mori cone of $X$, that is consists of
classes of effective curves.
In the first case the dual cone
is generated by the following
classes: $e_1-e_2,e_2-e_3,h-e_1,e_3$,
where $h-e_1$ is the class of the strict transform
of a line through the first point.
In the second case the dual cone is generated
by the classes $e_1-e_2,e_2,e_3,h-e_1,h-e_3$,
while in the third case it is generated
by the classes $e_1,e_2,e_3$
and $h-e_1,h-e_2,h-e_3$.
\end{proof}

\subsection{Moving cones}
We are now going to study the moving cones of the 
first four cubic elliptic varieties appearing in Table~\ref{types}.
\begin{proposition}\label{SSS3S}
If $X$ is of type $X_{S}$ or $X_{SSS}$,
then $\Mov(X)=\Nef(X)$.
\end{proposition}
\begin{proof}
Let $D$ be an effective divisor whose class 
does not lie in $\Nef(X)$. 
Hence $D\cdot C<0$ for some
curve $C$ which spans an extremal
ray of the Mori cone of $X$. 
We claim that the curves numerically 
equivalent to any such $C$ span a divisor.
Since this divisor must be contained into the stable
base locus of $D$ we get that the
class of $D$ does not belong to $\Mov(X)$
and this, together with the inclusion 
$\Nef(X)\subset\Mov(X)$ gives the thesis.

By the proof of Proposition~\ref{nefX} the
Mori cone of a variety of type $X_{S}$ is
generated by the following effective classes:
$e_1-e_2,e_2-e_3,h-e_1,e_3$. The curves 
numerically equivalent to these classes span
respectively $E_1-E_2$, $E_2-E_3$,
the strict  transform of the cubic cone 
$T_{p_1}Y\cap Y$ and
the exceptional divisor $E_3$.

The Mori cone of a variety of type $X_{SSS}$ is
generated by the following effective classes:
$e_1,e_2,e_3$ and $h-e_1,h-e_2,h-e_3$.
In these cases we obtain the divisors
$E_i$ and the strict transforms 
of the cubic cones $T_{p_i}Y\cap Y$ for
any $i\in\{1,2,3\}$, respectively.
\end{proof}

\begin{proposition}\label{3}
For any cubic elliptic variety $X$ of type
$X_3$ or $X_{S2}$, the moving cone is 
$\Mov(X)=\Nef(X)\cup\varphi^*\Nef(X)$,
where $\varphi$ is the flop of $X$
described in Proposition~\ref{floppic}.
\end{proposition}
\begin{proof}
Observe that the curves numerically
equivalent to one of the generators 
of the Mori cone of $X$ span either
a divisor or a variety of codimension two.
For both types $X_3$ and $X_{S2}$ the
only class which spans a variety of
codimension two is $h-e_1$.
Let $X$ be of type $X_3$ and 
consider the following cone of $A^1(X)\otimes\qq$
\begin{equation}\label{cone-1}
 \Cone(e_2-e_3,e_3,3h-2e_1-e_2,e_1-h).
\end{equation}
We claim that if $D$
is a movable non-nef class of $X$, then 
it belongs to the dual of this cone.
First of all since $D$ is not nef then it
has negative intersection with $h-e_1$.
The curves numerically equivalent to one of the
first two classes span divisors of $X$.
The same holds for the curves equivalent 
to $3h-2e_1-e_2$. Indeed consider the divisor 
linearly equivalent to $\pi^*\Osh_{\pp^{n-1}}(1)$
\[
 \pi^*(\pi_*(E_1-E_2))
 =
 (E_1-E_2) + (E_2-E_3) + V,
\]
where $V$ is the strict transform of the 
hyperplane section $T_{p_1}Y\cap Y$.
The fiber over a point of $\pi(E_1-E_2)$
has a component in $V$ whose class is
$3h-2e_1-e_2$ since
$3h-e_1-e_2-e_3=(e_1-e_2)+(e_2-e_3)+(3h-2e_1-e_2)$.
This proves the claim. 
To conclude we observe that the dual of the 
cone of~\eqref{cone-1} is $\varphi^*\Nef(X)$
and thus it is generated by movable classes.

Let $X$ be of type $X_{S2}$ and 
consider the following cone of $A^1(X)\otimes\qq$
\begin{equation}
\label{cone-2}
 \Cone(e_2,e_3,2h-e_1-e_2,3h-2e_1-e_3,e_1-h).
\end{equation}
As before we claim that if $D$
is a movable non-nef class of $X$, then 
it belongs to the dual of this cone. 
First of all since $D$ is not nef then it
has negative intersection with $h-e_1$.
The curves numerically equivalent to one of the
first two classes span divisors of $X$.
Concerning the third class, observe that the 
class of a fiber of $\pi$ is $3h-e_1-e_2-e_3$
and its push-forward in $Y$ is the class of the plane
cubic obtained intersecting $Y$ with a plane $\Pi$ 
containing the line $L$.
If we take the plane $\Pi$ to be tangent to $Y$ at the star point $p_3$,
then the cubic splits as the union of the line
and the conic corresponding to $h-e_3$ 
and $2h-e_1-e_2$ respectively.
When $\Pi$ moves, the curves equivalent to $2h-e_1-e_2$
span a prime vertical divisor.
If we now take a plane $\Pi$ tangent to $Y$ at $p_1$, 
the fiber decomposes as the sum of 
a curve in $e_1-e_2$ and one in $3h-2e_1-e_3$. 
As before, if we let $\Pi$ move, the curves
equivalent to $3h-2e_1-e_3$ span a prime 
vertical divisor. 
Since $D$ is movable then it must have 
non-negative intersection with the first four 
classes. 
This proves the claim.
To conclude we observe that the dual of the 
cone of~\eqref{cone-2} is $\varphi^*\Nef(X)$
and thus it is generated by movable classes.
\end{proof}

\subsection{Finitely generated Cox rings}
Recall that a $\qq$-factorial projective
variety is {\em Mori dream} if its Cox
ring is finitely generated~\cite{HuKe}.
We conclude the section by showing 
which cubic elliptic varieties appearing in 
Table~\ref{types} are Mori dream.

\begin{lemma}\label{cmw}
Let $\pi: X\to Z$ be a jacobian elliptic fibration 
between Mori dream spaces.
If the Cox ring of $X$ is finitely generated
then the Mordell-Weil group of $\pi$
is finite.
\end{lemma}
\begin{proof}
Let $X_\eta$ be the generic fiber of $\pi$,
let $\sigma: Z\to X$ be a section of $\pi$
and let $E = \sigma(Z)$. The Riemann-Roch
space $H^0(X,E)$ is one dimensional 
since
$E$ is effective and it can not move in
a linear series because it corresponds
to a point on the elliptic curve $X_\eta$.
Since $E$ is irreducible, any set of generators of the 
Cox ring of $X$ must contain a basis
of $H^0(X,E)$. Thus the Mordell-Weil
group of $\pi$ must be finite.
\end{proof}

\begin{theorem}\label{moridream}
Let $X$ be one of the cubic elliptic varieties
of Table~\ref{types}. Then the following
are equivalent:
\begin{enumerate}
\item the Cox ring of $X$ is finitely generated;
\item the Mordell-Weil group of $\pi: X\to\pp^{n-1}$ is finite.
\end{enumerate}
\end{theorem}
\begin{proof}
$(1)\Rightarrow (2)$. Follows from Lemma~\ref{cmw}.

$(2)\Rightarrow (1)$.
Since the Mordell-Weil group of $\pi$ is finite, 
looking at Table~\ref{types} we have that
$X$ must be either $X_3$, $X_{S}$, $X_{S2}$
or $X_{SSS}$.
In each of these cases we are going 
to use~\cite{HuKe}, showing
that the moving cone $\Mov(X)$
is union of finitely many polyhedral
chambers, each of which is pull-back
via a small $\qq$-factorial modification
$\phi: X\to X_i$ of $\Nef(X_i)$,
the last being generated by a
finite number of semiample
classes.

In the cases $X_{SSS}$ and $X_{S}$ 
the moving cone $\Mov(X)=\Nef(X)$ 
by Proposition~\ref{SSS3S}.
In cases $X_3$ and $X_{S2}$, by Proposition~\ref{3} 
the moving cone $\Mov(X)$ is the 
union of the two polyhedral chambers 
$\Nef(X)$ and $\varphi^*\Nef(X)$,
where $\varphi: X\to X$ 
is the small $\qq$-factorial modification
defined in Proposition~\ref{floppic}.
In all the cases we
conclude by Proposition~\ref{nefX}.
\end{proof}

\begin{remark}
Theorem~\ref{moridream} is the converse
of Lemma~\ref{cmw} for the cubic elliptic varieties
of Table~\ref{types}. The converse
of the lemma is not true in general:
given a jacobian elliptic fibration
$X\to Z$, with finite Mordell-Weil
group and $Z$ Mori dream, the
variety $X$ is not necessarily Mori dream.

For example consider the lattice 
$\Lambda= U\oplus 3A_1\oplus A_2$.
Since $\Lambda$ is an even hyperbolic 
lattice of rank $7\leq 10$, then it embeds into
the K3 lattice by~\cite{Ni}.
Thus by the global Torelli theorem there
exists a K3 surface $X$ whose Picard
lattice is isometric to $\Lambda$.
We observe that the surface $X$ admits
a jacobian elliptic fibration with finite
(indeed trivial) Mordell-Weil group
by~\cite[Table 1, n.19]{Sh}.
Moreover the automorphism group
of $X$ is infinite since the lattice 
$\Lambda$ is not $2$-elementary
and does not appear in the list
of~\cite[Theorem 2.2.2]{Do}. Hence
$X$ is not Mori dream by~
\cite[Theorem 2.7, Theorem 2.11]{AHL}.

We could not find an example of a 
variety $X$ which admits a unique
jacobian elliptic fibration $X\to Z$
with finite Mordell-Weil group,
$Z$ Mori dream and such that $X$ is 
not Mori dream as well.

\end{remark}

\newpage

\section{Cox rings}

In this section we provide a presentation
for the Cox rings of the cubic elliptic varieties
of type $X_{3}$, $X_{S}$, $X_{S2}$
and $X_{SSS}$.
Without loss of generality we can assume that
the defining polynomial of a smooth cubic hypersurface
is one of the polynomials listed in Table~\ref{types}.
Here we require an extra
assumption in cases $X_3$ and $X_S$:
the coefficient of $T_n^3$ is non-zero.
This condition will be needed in the 
proof of Lemma~\ref{vample}.

\begin{center}
\footnotesize
\begin{longtable}{lll}
\hline
Type 
&
Equations for $Y\subset\pp^{n+1}$ and the line $L$
&
Conditions on the polynomials
\\
\hline
\\
$X_{3}$
&
$
\begin{array}{l}
\scriptsize
T_{n+1}(a'+T_{n+2}\,a_1)+T_{n+2}\,b'+b_1=0
\\
\\
\scriptsize
T_1=\dots=T_{n-1}=T_{n+1}=0
\end{array}
$
&
$
\begin{array}{l}
\scriptsize
b', b_1\in\cc[T_1,\dots,T_n]
\\
\scriptsize
a'\in\cc[T_1,\dots,T_{n+1}]
\\
\scriptsize
a_1\in\cc[T_1,\dots,T_{n+2}]
\\
\scriptsize
\text{The coefficient of $T_n^2$
in $b'$ is zero}
\\
\scriptsize
\text{The coefficient of $T_n^3$
in $b_1$ is non-zero}
\end{array}
$
\\
\\
\hline
\\
$X_{S}$
&
$
\begin{array}{l}
\scriptsize
T_{n+1}\,a_2+b_2=0
\\
\\
\scriptsize
T_1=\dots=T_{n-1}=T_{n+1}=0
\end{array}
$
&
$
\begin{array}{l}
\scriptsize
b_2\in\cc[T_1,\dots,T_n]
\\
\scriptsize
a_2\in\cc[T_1,\dots,T_{n+2}]
\\
\scriptsize
\text{The coefficient of $T_n^3$
in $b_2$ is non-zero}
\end{array}
$
\\
\\
\hline
\\
$X_{S2}$
&
$
\begin{array}{l}
\scriptsize
T_{n+1}\,a_3+b_3=0
\\
\\
\scriptsize
T_1=\dots=T_{n}=0
\end{array}
$
&
$
\begin{array}{l}
\scriptsize
b_3\in\cc[T_1,\dots,T_n]
\\
\scriptsize
a_3\in\cc[T_1,\dots,T_{n+2}]
\\
\scriptsize
\text{The coefficient of $T_{n+1}^2$ in $a_3$ is zero}
\\
\scriptsize
\text{The coefficient of $T_{n+1}T_{n+2}$ in $a_3$ 
is zero}
\end{array}
$
\\
\\
\hline
\\
$X_{SSS}$
&
$
\begin{array}{l}
\scriptsize
T_{n+1}\,T_{n+2}\,a_4+b_4=0
\\
\scriptsize
T_1=\dots=T_{n}=0
\end{array}
$
&
$
\begin{array}{l}
\scriptsize
b_4\in\cc[T_1,\dots,T_n]
\\
\scriptsize
a_4\in\cc[T_1,\dots,T_{n+2}]
\end{array}
$
\\
\\
\hline
\\
\caption{Equations of extremal cubic elliptic varieties}
\label{cox-types}
\end{longtable}
\end{center}

We now define four homomorphisms
of rings which will be used in
Theorem~\ref{teor}.

\begin{center}
\footnotesize
\begin{longtable}{ll}
\hline
Homomorphism & Defined by
\\
\hline
\\
\begin{tabular}{l}
$
\beta_1:
\cc[T_1,\dots,T_{n+3}]
\to
\cc[T_1,\dots,T_{n+3},S_1,S_2,S_3]
$\\
\end{tabular}
&
$
 \begin{array}{ll}
 T_k & \mapsto T_k\,S_1\,S_2^2\,S_3^3\\[2pt]
 T_n & \mapsto T_n\,S_1\,S_2\,S_3\\[2pt]
 T_{n+1} & \mapsto T_{n+1}\,S_1^2\,S_2^3\,S_3^3\\[2pt]
 T_{n+2} & \mapsto T_{n+2}\\[2pt]
 T_{n+3} & \mapsto T_{n+3}\,S_1^3\,S_2^3\,S_3^3
\end{array}
$
\\
\\
\hline
\\
\begin{tabular}{l}
$
\beta_2:
\cc[T_1,\dots,T_{n+2}]
\to
\cc[T_1,\dots,T_{n+2},S_1,S_2,S_3]
$\\
\end{tabular}
&
$
 \begin{array}{ll}
 T_k & \mapsto T_k\,S_1\,S_2^2\,S_3^3\\[2pt]
 T_n & \mapsto T_n\,S_1\,S_2\,S_3\\[2pt]
 T_{n+1} & \mapsto T_{n+1}\,S_1^3\,S_2^3\,S_3^3\\[2pt]
 T_{n+2} & \mapsto T_{n+2}
\end{array}
$
\\
\\
\hline
\\
\begin{tabular}{l}
$
\beta_3:
\cc[T_1,\dots,T_{n+3}]
\to
\cc[T_1,\dots,T_{n+3},S_1,S_2,S_3]
$\\
\end{tabular}
&
$
 \begin{array}{ll}
 T_k & \mapsto T_k\,S_1\,S_2^2\,S_3^3\\[2pt]
 T_n & \mapsto T_n\,S_1^2\,S_2^2\,S_3\\[2pt]
 T_{n+1} & \mapsto T_{n+1}\,S_3^3\\[2pt]
 T_{n+2} & \mapsto T_{n+2}\,S_1\,S_2\\[2pt]
 T_{n+3} & \mapsto T_{n+3}\,S_1^3\,S_2^6
\end{array}
$
\\
\\
\hline
\\
\begin{tabular}{l}
$
\beta_4:
\cc[T_1,\dots,T_{n+3}]
\to
\cc[T_1,\dots,T_{n+3},S_1,S_2,S_3]
$\\
\end{tabular}
&
$
 \begin{array}{ll}
 T_k & \mapsto T_k\,S_1\,S_2^2\,S_3^3\\[2pt]
 T_n & \mapsto T_n\,S_1\,S_2^2\,S_3^3\\[2pt]
 T_{n+1} & \mapsto T_{n+1}\,S_1^3\\[2pt]
 T_{n+2} & \mapsto T_{n+2}\,S_2^3\\[2pt]
 T_{n+3} & \mapsto T_{n+3}\,S_3^3
\end{array}
$
\\
\\
\hline
\\
\end{longtable}
\end{center}

The following theorem is the 
main result of this section.
We postpone its proof until the
end of the section.

\begin{theorem}\label{teor}
Let $Y$ and $L$ be a smooth cubic hypersurface
and a line of $\pp^{n+1}$
whose defining equations
are given in Table~\ref{cox-types}.
Let $X$ be the corresponding cubic 
elliptic variety of type $X_3$, $X_S$, 
$X_{S2}$, $X_{SSS}$. Then the Cox ring 
of $X$ is one of the following.
\begin{enumerate}
\item Type $X_3$: the Cox ring is
$\cc[T_1,\ldots, T_{n+3},S_1,S_2,S_3]/\mathfrak{I}_1$,
where $\mathfrak{I}_1$ is generated by
{\small
\[
  \frac{\beta_1(T_{n+3}-T_{n+1}\,a_1-b')}{S_1^2S_2^3S_3^3}
  \qquad
 \frac{\beta_1(T_{n+2}\,T_{n+3}+T_{n+1}\,a'+b_1)}{S_1^3S_2^3S_3^3}
\]
}
with the $\zz^4$-grading given by the 
grading matrix
{\scriptsize
\[
 \left[
 \begin{array}{rrrrrrrrrr}
   1&\cdots& 1 & 1 & 1  & 1 & 2& 0 & 0 & 0\\
  -1&\cdots&-1 &   -1 & -2 & 0 &-3& 1& 0 & 0\\
  -1&\cdots&-1 &  0 & -1 & 0 & 0 & -1&1 & 0\\
  -1&\cdots&-1 &  0 & 0  & 0 & 0 & 0 &-1& 1
 \end{array}
 \right]
\]
}
\item
Type $X_S$: the Cox ring is
$\cc[T_1,\ldots, T_{n+2},S_1,S_2,S_3]/\mathfrak{I}_2$,
where $\mathfrak{I}_2$ is generated by
{\small
\[
  \frac{\beta_2(T_{n+1}\,a_2+b_2)}{S_1^3S_2^3S_3^3}
\]
}
with the $\zz^4$-grading given by the 
grading matrix
{\scriptsize
\[
 \left[
 \begin{array}{rrrrrrrrrr}
   1&\cdots& 1 &   1 & 1  & 1 & 0& 0 & 0\\
  -1&\cdots&-1 & -1 & -3 & 0 &1 & 0 & 0\\
  -1&\cdots&-1 &  0 & 0  & 0 & -1& 1& 0\\
  -1&\cdots&-1 &  0 & 0  & 0 & 0 &-1& 1
 \end{array}
 \right]
\]
}
\item
Type $X_{S2}$: the Cox ring is
$\cc[T_1,\ldots, T_{n+3},S_1,S_2,S_3]/\mathfrak{I}_3$,
where $\mathfrak{I}_3$ is generated by
{\small
\[
 \frac{\beta_3(T_{n+3}-a_3)}{S_1^2S_2^2}
 \qquad
 \frac{\beta_3(T_{n+1}\,T_{n+3}+b_3)}{S_1^3S_2^6S_3^3}
\]
}
with the $\zz^4$-grading given by the 
grading matrix
{\scriptsize
\[
 \left[
 \begin{array}{rrrrrrrrrr}
   1&\cdots& 1 &   1 & 1  &  1 & 2&  0  & 0 & 0\\
  -1&\cdots&-1 & -2 & 0  & -1 &-3&  1 &  0 & 0\\
  -1&\cdots&-1 &  0 & 0  &  0 & -3& -1 & 1 & 0\\
  -1&\cdots&-1 &  -1 & -3  & 0 & 0 &  0 & 0 & 1
 \end{array}
 \right]
\]
}
\item
Type $X_{SSS}$: the Cox ring is
$\cc[T_1,\ldots, T_{n+3},S_1,S_2,S_3]/\mathfrak{I}_4$,
where $\mathfrak{I}_4$ is generated by
{\small
\[
 \beta_4(T_{n+3}-a_4)
 \qquad
 \frac{\beta_4(T_{n+1}\,T_{n+2}\,T_{n+3}+b_4)}{S_1^3S_2^3S_3^3}
\]
}
with the $\zz^4$-grading given by the 
grading matrix
{\scriptsize
\[
 \left[
 \begin{array}{rrrrrrrrrr}
   1&\cdots& 1 &  1 &  1  & 1  & 0&  0 & 0\\
  -1&\cdots&-1 & -3 & 0  & 0  & 1&  0 & 0\\
  -1&\cdots&-1 &  0 & -3 & 0  & 0 & 1 & 0\\
  -1&\cdots&-1 &  0 &  0 & -3 & 0 & 0 & 1
 \end{array}
 \right]
\]
}
\end{enumerate}
\end{theorem}

\vspace{1mm}

\subsection{Algebraic preliminaries}
\label{diag}
We follow the construction
given in Section 3 of~\cite{BHK}.
Let $\cc[T,S]$ be a polynomial
ring in 
the variables $T_1,\dots,T_r,
S_1,\dots,S_s$, graded by an abelian
group $K_T\oplus K_S$.
Let $\cc[T,S^{\pm 1}]$
be its localization with respect to 
all the $S$ variables and let
$\cc[T]$ be the polynomial ring
in the first $r$ variables graded
by $K_T$. Denote by
$\cc[T,S^{\pm 1}]_0$
the degree zero part of 
$\cc[T,S^{\pm 1}]$ with respect
to the $K_S$ grading.
Assume the following diagram of 
homomorphisms is given:
\begin{equation}\label{pre}
 \xymatrix{
  I\ar[r] & 
  \cc[T,S]\ar[rr]^-{\psi}
  \ar[d]_-j
  && R\\
   & 
  \cc[T,S^{\pm 1}]
  &&\cc[T]\ar[u]^-{\rho}
  \ar[llu]_-{\beta}\ar[lld]_{\alpha}^-\cong\\
  & \cc[T,S^{\pm 1}]_0
  \ar[u]^-{j_0}
  && J\ar[u]
 }
\end{equation}
such that $R$ is a $K_T\oplus K_S$-graded
domain, $\psi$ is a graded surjective 
homomorphism with kernel $I$,
$\rho$ a graded 
homomorphism with kernel $J$,
both $j$ and $j_0$ are inclusions,
$\rho=\psi\circ\beta$ and
$\alpha(T_i) = T_i\cdot m_i(S)
=\beta(T_i)$, for any $i$, where
$m_i(S)\in\cc[T,S]$ is a monomial
in the variables $S$.

\begin{proposition}\label{ideals}
Under the above assumptions
let $J'\subset\cc[T,S]$ be the 
extension and contraction of the
ideal $\alpha(J)$. Then $J'\subset I$.
\end{proposition}
\begin{proof}
Observe that $\beta(J)\subset I$ 
since $\rho=\psi\circ\beta$.
Moreover from
\[
 \beta(J)\cdot\cc[T,S^{\pm 1}]
 =
 \alpha(J)\cdot\cc[T,S^{\pm 1}],
\]
we get that $J'$ is contained
in the saturation of $I$ with
respect to the variables $S$.
Since $R$ is a domain, then $I$
is saturated,
hence we get the statement.
\end{proof}

The following statement identifies a Cox ring 
with certain subalgebras. 
Consider a factorially $K$-graded 
normal affine algebra $R = \oplus_K R_w$ 
with pairwise non-associated $K$-prime generators
$f_1, \ldots, f_r$ and set $w_i := \deg(f_i) \in K$.
The $K$-grading is {\em almost free} if any 
$r-1$ of the $w_i$ generate $K$ as a group.
The {\em moving cone\/} 
$\Mov(R) \subseteq K_{\mathbb{Q}}$ 
is the intersection over all cones
in $K_{\mathbb{Q}}$ generated 
by any $r-1$ of the degrees $w_i$.
Recall that $\Mov(R)$ comes with a subdivision 
into finitely many polyhedral {\em GIT-cones\/}
$\lambda(w)$ associated to the classes 
$w \in \Mov(R)$, see~\cite[Prop.~3.9]{Ha2}.

\begin{proposition}\label{criterion}
Let $X$ be a $\mathbb{Q}$-factorial projective variety 
with finitely generated Cox ring $\mathcal{R}(X)$
and $R \subseteq \mathcal{R}(X)$ a 
finitely generated normal almost freely 
factorially $\Cl(X)$-graded subalgebra
such that $R$ and $\mathcal{R}(X)$ have 
the same quotient field.
If there is a very ample divisor $D$ on $X$ 
such that $R_{[D]} = \mathcal{R}(X)_{[D]}$
holds and $\lambda([D]) \subseteq \Mov(R)$
is of full dimension, then we have $R = \mathcal{R}(X)$.
\end{proposition}

\begin{proof}
Consider the total coordinate space
$\overline{X} := {\rm Spec} \, \mathcal{R}(X)$
and $\overline{Y} := {\rm Spec} \, R$. 
Both come with an action of the characteristic 
quasitorus $H := {\rm Spec} \, \mathbb{C}[\Cl(X)]$ 
and we have a canonical $H$-equivariant morphism 
$\overline{X} \to \overline{Y}$.
Moreover, for $w := [D] \in \Cl(X)$, the inclusion
$R(w)  \subseteq \mathcal{R}(X)(w)$ defines 
a morphism $\overline{X}(w) \to \overline{Y}(w)$.
Altogether we arrive at a commutative diagram
$$ 
\xymatrix{
&&&
\\
{\widehat{X}}
\ar[d]_{/ H(w)}^{q_X}
\ar@/^2pc/[rrrr]
\ar@{}[r]|\subseteq
&
{\overline{X}}
\ar[rr]
\ar[d]
& &
{\overline{Y}}
\ar[d]
& 
{\widehat{Y}}
\ar[d]^{/\!\!/ H(w)}_{q_Y}
\ar@{}[l]|\supseteq
\\
{\widehat{X}}(w)
\ar@{}[r]|\subseteq
\ar[dr]_{/\mathbb{C}^*}
&
{\overline{X}(w)}
\ar[rr]
& &
{\overline{Y}(w)}
&
{\widehat{Y}}(w)
\ar@{}[l]|\supseteq
\ar[dl]^{/\mathbb{C}^*}
\\
&
X
\ar[rr]
&
&
Y
&
}
$$
Here $\widehat{X} \subseteq \overline{X}$ and 
$\widehat{Y} \subseteq \overline{Y}$ are the 
respective unions of all localizations 
$\overline{X}_f$ and~$\overline{Y}_f$, where 
$f$ is of degree $w$, and the subsets 
$\widehat{X}(w) \subseteq \overline{X}(w)$ and 
$\widehat{Y}(w) \subseteq \overline{Y}(w)$
are defined analogously.
The downwards maps are quotients with respect to 
the action of the subgroup $H(w) \subseteq H$ 
corresponding to the map of character groups 
$\Cl(X) \to \Cl(X) / \mathbb{Z}w$.
Note that by ampleness of $D$,  the composition 
$\widehat{X} \to X$ is the characteristic space.

Since $R$ is almost freely factorially 
$\Cl(X)$-graded and $w$ lies in the relative 
interior of $\Mov(R)$, we infer 
from~\cite[Thm.~3.6]{Ha2} that also 
$\widehat{Y} \to Y$ 
is a characteristic space.
The resulting variety $Y$ is 
projective~\cite[Prop.~3.9]{Ha2}.
As a dominant morphism of projective varieties,
the induced map $X \to Y$ is surjective.
Since the GIT-cone $\lambda(w)$ is of full dimension,
the fibers of $\widehat{Y} \to Y$ are precisely 
the $H$-orbits, use~\cite[Lemma~3.10]{Ha2}.
The commutative diagram then yields that 
the $H$-equivariant morphism
$\widehat{X} \to \widehat{Y}$ is surjective.
Moreover, the complement 
$\overline{Y} \setminus \widehat{Y}$ is 
of codimension at least two in $\overline{Y}$, 
see~\cite[Constr.~3.11]{Ha2}.
Thus, by Richardson's Lemma, the 
birational morphism 
$\overline{X} \to \overline{Y}$ 
of normal affine varieties is an isomorphism. 
The assertion follows.
\end{proof}

\subsection{Proof of Theorem~\ref{teor}}
Let us give here all the necessary preliminary
lemmas to prove the main result of the section.
For each cubic elliptic variety
$X$ in Theorem~\ref{teor}
we construct the $\zz^4$-graded ring
\[
 R_n:=\cc[T,S]/\mathfrak{I},
\]
where $\mathfrak{I}$ is one of the four
ideals $\mathfrak{I}_1,\dots,\mathfrak{I}_4$.
Consider a factorially $K$-graded 
normal affine algebra $R = \oplus_K R_w$ 
with pairwise non-associated $K$-prime generators
$f_1, \ldots, f_r$ and set $w_i := \deg(f_i) \in K$.
\begin{remark}\label{irr}
Given an effective divisor $D$ we 
consider the subspace $V$ of $H^0(X,D)$,
generated by all the sections corresponding 
to reducible divisors. 
Observe that any system of generators of the Cox ring
of $X$ must contain all the elements of a basis
of $H^0(X,D)$ which are not in $V$.
\end{remark}

\begin{lemma}\label{subring}
Let $R_n$ be as before
Then the following hold.
\begin{enumerate}
\item $R_n$ is a subalgebra of $\R(X)$.
\item $R_n$ and $\R(X)$ have the same
quotient field.
\item $R_n$ is almost free factorially graded.
\end{enumerate}
\end{lemma}

\begin{proof}
Each grading
matrix $Q_n$ is written with respect
to the basis $(H,E_1,E_2,E_3)$.
We are going to show that the columns
of any such matrix are degrees
of generators of the Cox ring according
to Remark~\ref{irr}.
We will proceed in two steps.
First of all we will construct 
in each case an homomorphism of rings:
\[
 \psi: \cc[T,S]\to\R(X)
\]
which maps the generators 
$T_j$ and $S_k$ to certain
sections of the Cox ring.
Then we will show that the kernel of $\psi$
is the defining ideal of $R_n$.

Any irreducible divisor of Riemann-Roch 
dimension one is a generator of the Cox ring.
Among these there are the exceptional divisors
corresponding to the last three columns 
of each matrix $Q_n$.
We also have the strict transforms of the intersections
of $Y$ with a tangent hyperplane, corresponding to 
the columns whose first entry is 1, and at least one of the 
others entries is smaller than $-1$.
Moreover, by Proposition~\ref{floppic} the classes $[2,-3,0,0]$ of type $X_3$ and
$[2,-3,-3,0]$ of type $X_{S2}$ are the flop images of
$[1,-2,-1,0]$ and $[1,-2,0,-1]$ respectively.

We now claim that if $D$ is an effective irreducible divisor 
with class $H-m_1E_1-m_2E_2-m_3E_3$, then 
either $m_i\leq 1$ for each $i=1,2,3$ or $D$ is
the intersection of $Y$ with a tangent hyperplane.
Indeed if the three points are distinct then the
claim is obvious. Otherwise let us assume for example
that $p_2$ is a point of the exceptional divisor
on $p_1$. Then $e_1-e_2$ is the class in $A_1(X)$
of a fiber of the $\pp^1$-bundle $E_2-E_1$ and 
$D\cdot(e_1-e_2)=m_1-m_2\geq 0$ since $D$ is irreducible
and distinct from $E_1-E_2$. Hence the biggest multiplicities
are those of the points in $L\cap Y$ and the claim follows.

Since we already considered the sections of $Y$
with a tangent hyperplane, 
from the previous claim we now concentrate on the case 
in which all the $m_i$ are less than or equal to 1.
Observe that $H^0(X,\pi^*\Osh_{\pp^{n-1}}(1))$
contains no reducible sections when $X$
is of type $X_{SSS}$ and 
just one reducible section for the 
remaining three types. Hence
by Remark~\ref{irr} we get the columns of degree
$[1,-1,-1,-1]$ 
(they are $n$ 
in type $X_{SSS}$ and $n-1$ otherwise).
Moreover, when $X$ is of type 
$X_3$, $X_S$ or $X_{S2}$, the
Riemann-Roch dimension of a divisor of
degree $[1,-1,0,0]$ is $n+1$, while with the previous 
generators one can only form a $n$-dimensional
subspace. Hence, again by Remark~\ref{irr}
we add a generator in this degree and a similar
argument applies to $[1,0,0,0]$ for $X_3$ and $X_S$.

We have thus defined the homomorphism
$\psi$. Since we are considering four
cases, let us denote by $\psi_i$, for
$i\in\{1,\dots,4\}$, these homomorphisms.
By the definition of $\beta_i$, the homomorphism 
$\rho_i:=\psi_i\circ\beta_i$, 
is just the composition of the natural map
$\cc[T]\to\R(Y)$
with the pull-back map $\R(Y)\to\R(X)$.
If we denote by $J_i$ the kernel of $\rho_i$, we have that 
\[
\begin{array}{rcl}
J_1  & = & \langle
  T_{n+3} -T_{n+1}\,a_1-b',\,
  T_{n+2}\,T_{n+3}+T_{n+1}\,a'+b_1
 \rangle,\\[3pt]
J_2 & = &\langle T_{n+1}\,a_2+b_2\rangle,\\[3pt]
J_3 & = & \langle
  T_{n+3} - a_3,\,
  T_{n+1}\,T_{n+3} - b_3
 \rangle,\\[3pt]
J_4  &  = &
 \langle
  T_{n+3} - a_4,\,
  T_{n+1}\,T_{n+2}\,a_4+b_4
 \rangle.
 \end{array}
\]
By the generality assumptions on the polynomials
$a_i,b_i,c_i$ and $d_i$ we have that each $J_i$
is prime.
For each of the four cases we now refer 
to diagram~\eqref{pre} where the ring
$R$ in the diagram is the image
of $\psi_i$.
By Proposition~\ref{ideals} the
contraction and extension $J'_i$
of the ideal $\alpha_i(J_i)$ is contained
in $I_i:=\ker(\psi_i)$. 
By~\cite[Proposition 3.3]{BHK} and
the fact that $J_i$ is prime, we deduce that
also $J'_i$ is prime. 
We are now going to prove that 
\[
 \mathfrak{I}_i = J'_i
 \quad
 \text{for $i\in\{1,\dots,4\}$},
\]
where $\mathfrak{I}_i$ is the $i$-th ideal 
appearing in Theorem~\ref{teor}.
This is equivalent to show
that each ideal $\mathfrak{I}_i$ is saturated with respect to
the variables $S$. For $\mathfrak{I}_2$ this is
straightforward since it is principal and the generator
is irreducible.
In the remaining cases, since each ideal 
$\mathfrak{I}_i$ has two generators
it is enough to prove that there are no components
of codimension one in $V(S_1\,S_2\,S_3)$. 
The second generator of $\mathfrak{I}_4$
is a polynomial in the $T_j$ and hence there is nothing
to prove. The second generator of $\mathfrak{I}_1$
can be written as 
$f:=T_{n+2}\,T_{n+3}+\beta_1(T_{n+1}a'+b_1)S_1^{-3}S_2^{-3}S_3^{-3}$.
The first monomial is $T_{n+2}\,T_{n+3}$, while
the sum of the remaining monomials does not
contain these two variables and does not vanish identically 
on $V(S_i)$. Hence $V(f,S_i)$ is irreducible and since
$f$ does not divide the first generator, then there are
no components of codimension 1 in $V(S_1\,S_2\,S_3)$.
A similar analysis applies to $\mathfrak{I}_3$.

We proved that each $R_n=\cc[T,S]/\mathfrak{I}_i
=\cc[T,S]/J'_i$ is a domain.
Moreover $J_i'\subset I_i$
implies that $R\subset R_n$ and we claim
that $R_n= R$. By construction
we know that $\dim R_n=n+4$. 
Observe now that the field of rational 
functions of $Y$ has dimension $n$ and is equal to 
the field of homogeneous fractions of $R$. Since $R$ 
is graded by $\zz^4$ we conclude
that also $\dim R=n+4$.
Moreover $R$ is a domain too
since it is contained in $\R(X)$.
We conclude by observing that
$R$ and $R_n$ are domains 
of the same dimension and hence the inclusion 
$R\subset R_n$ implies that $R=R_n$.
This proves (1).

Part (2) of the statement follows from
the fact that both $R_n$ and $\R(X)$ contain
the homogeneous coordinate ring of 
the cubic hypersurface $Y$ as a subring.

According to~\eqref{pre} the ideal $\mathfrak{I}$
of $\cc[T,S]$ is obtained by extending and 
contracting the homogeneization $\alpha(J)$
of the ideal $J$. Hence $R_n$ is factorially
graded by~\cite[Theorem 3.2]{BHK} and
it is almost free graded since by~\cite[Corollary 3.4]{BHK}
it is the Cox ring of a toric ambient modification of $Y$.
This proves (3).

\end{proof}

According to Lemma~\ref{subring} the
algebra $R_n$ is a subalgebra of
the Cox ring $\R(X)$. 
Let $f_1$ be the generator of $R_n$ 
corresponding to the variable $T_1$
and let $D$ be the divisor of $X$ 
defined by $f_1$. In what follows
with abuse of notation we will denote by 
the same symbol the divisor $D$ and its 
support.

\begin{lemma}\label{rinascente}
The following properties hold:
\begin{enumerate}
\item $R_{n-1}$ is isomorphic to 
$R_n/\langle T_1\rangle$ for any $n>3$;
\item $D$ is a cubic elliptic variety of the same type
of $X$, of dimension $n-1$;
\item there is a surjective morphism 
$\R(D)\to\R(D')$, where $D'$ is the image of
$D$ via some composition of the $\sigma_i$.
\end{enumerate}
\end{lemma}
\begin{proof}
(1) follows directly from the definition of $R_n$, 
while (2) is implied by the fact that $D$ is 
the pull-back of a hyperplane of $\pp^{n-1}$ 
via the elliptic fibration $\pi: X\to\pp^{n-1}$.
(3) follows from the fact that each composition
of the $\sigma_i$ is a blow-down and then it 
is a toric ambient modification in the sense 
of~\cite[Remark 3.6]{BHK}.
\end{proof}

\begin{lemma}\label{vample}
Let $X$ be a cubic elliptic $n$-dimensional
variety of type $X_3$, $X_S$, $X_{S2}$ 
or $X_{SSS}$ and let $W=4H-3E_1-2E_2-E_3$.
Then the following hold.
\begin{enumerate}
\item The divisor $W$ is very ample.
\item The GIT chamber $\lambda([W])
\subseteq\Mov(R_n)$ is full-dimensional.
\end{enumerate}
\end{lemma}
\begin{proof}
We begin by proving (1).
Writing $W$ as
\[
 W=(H-E_1-E_2-E_3) + (H-E_1-E_2) + (H-E_1) + H,
\]
we observe that it is ample since 
it lies in the interior of the nef cone
of $X$ by Proposition~\ref{nefX}.
The linear system $|W|$ is base point
free since all the summands in the above 
sum are base point free.
Finally the morphism defined by the linear
system $|H|$ is birational, since it is just
the contraction $X\to Y$. Hence $|W|$
is an ample and spanned linear system
which defines a degree one morphism
and thus $W$ is very ample.

Let $w_i$ be the degree of the $i$-th
generator of $R_n$, that is the $i$-th
column of the corresponding grading 
matrix given in Theorem~\ref{teor}.
A direct calculation shows that
the class $w$ of $W$ is not contained
in any two-dimensional cone spanned
by the $w_i$. The three-dimensional
cones ${\rm cone}(w_\alpha,w_\beta,w_\gamma)$
which contain $w$ into their relative 
interiors correspond to the sets 
of indices
$I=\{\alpha,\beta,\gamma\}$ given
in the table below, where $i\in\{1,\dots,n-1\}$
and $f_j$ denotes the $j$-th generator
of the ideal $\mathfrak{I}$ given in 
Theorem~\ref{teor}.

\begin{center}
\footnotesize
\begin{tabular}{c|c|c|c}
 $X_3$ & $\{i,n+1,n+2\}$ & $\{i,n,n+4\}$ & $\{n+3,n+4,n+5\}$\\[5pt]
 & $f_1^I = T_{n+1}T_{n+2}^2$ & $f_2^I = T_n^3$ & $f_1^I=T_{n+3}S_1$\\[5pt]
 \midrule
 $X_S$ & $\{i,n,n+3\}$ & & \\[5pt]
 & $f_1^I = T_n^3$ & & \\[5pt]
 \midrule
 $X_{S2}$ & $\{i,n+2,n+4\}$ & $\{n,n+2,n+4\}$\\[5pt]
 & $f_1^I = T_{n+2}^2$ & $f_1^I = T_{n+2}^2$ & \\[5pt]
 \end{tabular}
\end{center}
Let $\mathfrak{I}$ be the 
ideal of relations of $R_n$.
Let $T_{n+2+i}=S_i$ if $X$ is of type $X_3$ and
$T_{n+3+i}=S_i$ otherwise.
For any subset of indices
$I$ define the ideal
$\mathfrak{I}^I= 
\langle f^I : f\in\mathfrak{I}\rangle$,
where
\[
 f^I
 =
 f(U_1,\dots,U_{n+5})
 \quad
 \text{ where }
 \quad
 U_k
 =
 \begin{cases}
  T_k &\text{ if }k\in I\\
  0 & \text{ otherwise.}
 \end{cases}
\]
For any set of indices $I$ in the above
table, the ideal $\mathfrak{I}^I$ 
contains the monomial $f^I$
due to our assumption on the
defining equations written in 
Table~\ref{cox-types}.
This allows us to conclude that
the corresponding cone 
${\rm cone}(w_\alpha,w_\beta,w_\gamma)$
with $I=\{\alpha,\beta,\gamma\}$
is not an orbit cone.
Since $\lambda(w)$ is the intersection
of all the orbit cones which
contain $w$ into their relative interior
and since all such cones are full-dimensional
then we conclude that $\lambda(w)$
is full-dimensional as well.
\end{proof}

\begin{lemma}\label{case3}
Let $X$ be a cubic elliptic threefold of
type $X_3$, $X_S$, $X_{S2}$
or $X_{SSS}$. Then the Cox ring 
of $X$ is isomorphic to $R_3$.
\end{lemma}
\begin{proof}
Denote by $A$ the polynomial ring $\cc[T,S]$. 
If $X = X_{S}$, then a presentation of $R_3$ is
the Koszul complex:
\[
 \xymatrix{
  0\ar[r] & A(-w_1)\ar[r]
  & A\ar[r] & 0,
 }
\]
where $w_1=[3,-3,0,0]$ is the degree
of the generator of $\mathfrak{I}_2$.
If $X$ is one of the remaining types
then a presentation of $R_3$ is the Koszul 
complex:
\[
 \xymatrix{
  0\ar[r] & A(-w_1-w_2)\ar[r]
  & A(-w_1)\oplus A(-w_2)\ar[r] &
  A\ar[r] & 0,
 }
\]
where $w_1$ and $w_2$ are the degrees
of the generators of the ideal $\mathfrak{I}_i$
for $i\in\{1,3,4\}$.
A computer calculation done by using
the previous exact sequences shows that the 
dimension of the degree $w$ part of $R_3$
is $66$, $53$, $64$ and $75$ for the types
$X_3$, $X_S$, $X_{S2}$ and $X_{SSS}$ 
respectively.
In each case this dimension equals the 
Riemann-Roch dimension of the
class $w$. Hence $(R_3)_w=\R(X)_w$
and we conclude by Lemma~\ref{subring},
Lemma~\ref{vample} and
Proposition~\ref{criterion}.
\end{proof}

\begin{proof}[Proof of Theorem~\ref{teor}]
We proceed by induction on $n$.
The case $n=3$ is proved in Lemma~\ref{case3}.
Assume $n>3$. 
Observe that $H-E_1-E_2-E_3$ is linearly
equivalent to the divisor $D$ of $X$ 
defined by $f_1$ and that its  
push-forwards via 
$\sigma=\sigma_1\circ\sigma_2\circ\sigma_3$, 
$\sigma_2\circ\sigma_3$ and $\sigma_3$ 
equal those of $H$, 
$H-E_1$ and $H-E_1-E_2$ respectively.
According to 
Proposition~\ref{criterion}
it is enough to show that the degree $w$
part of $R_n$ equals that of $\R(X)$.
To this aim we consider the exact sequence
\[
 \xymatrix@1{
  0\ar[r] &
  H^0(X,W-D)\ar[r]^-{\cdot f_1} &
  H^0(X,W)\ar[r] &
  H^0(D,W|_D)\ar[r]\ar@/^15pt/[l]^\gamma &
  0,
 }
\]
where the last $0$ is due to Kawamata-Viehweg 
and the fact that $W-D-K_{X}$ is nef and big.
By the induction hypothesis and
our choice of $D$ we have
a surjective map $R_n\to R_{n-1}=\R(D)$.
This allows us to construct a section
$\gamma$ whose image is contained 
in $R_n$.

We claim that any section of $H^0(X,W-D)$
is in $R_n$ and this is enough to conclude.
The divisor $W-D=(H-E_1-E_2) + (H-E_1) + H$
is the pull-back of a divisor $W_2$ of $Y_2$.
Denote by $D_2$ the divisor of $Y_2$ which
is the image of $D$ via $\sigma_3$.
As before there is an exact sequence
\[
 \xymatrix{
  0\ar[r] &
  H^0(Y_2,W_2-D_2)\ar[r] &
  H^0(Y_2,W_2)\ar[r] &
  H^0(D_2,W_2|_{D_2})\ar[r]\ar@/^15pt/[l]^{\gamma_2} &
  0,
 }
\]
where the last $0$ is due to Kawamata-Viehweg and the fact that
$W_2-D_2-K_{Y_2}$ is linearly equivalent
to ${\sigma_3}_*(H-E_1+H)-K_{Y_2}$ which is nef and big.
By Lemma~\ref{rinascente} and the fact
that $\sigma_3: X\to Y_2$ is a toric ambient 
modification we get the following
diagram, where all the maps but the inclusion
$R_n\to\R(X)$ are surjective.
\[
 \xymatrix{
  R_n\ar[r]\ar@{->>}[d] & \R(X)\ar@{->>}[r] & \R(Y_2)\\
  R_{n-1} \ar@{=}[r] & \R(D) \ar@{->>}[r] & \R(D_2)
 }
\]
This allows us to construct a section
$\gamma_2: \R(D_2)_{w_2}\to \R(Y_2)_{w_2}$
whose image is contained in the image
of $R_n$.
Now we proceed in a similar way with the divisor
$W_2-D_2={\sigma_3}_*(2H-E_1)$
obtaining the divisors $W_1 = 
(\sigma_2\circ\sigma_3)_*(2H-E_1)$ and
$D_1 = (\sigma_2\circ\sigma_3)_*(H-E_1)$,
so that $W_1-D_1$ is pull-back of 
the divisor $\sigma_*(H)$ on $Y$. 
This last divisor is a hyperplane
section of $Y$ and thus a Riemann-Roch basis
consists of elements of the coordinate ring of $Y$
which is a homomorphic image of $R_n$.
This proves the claim.
Hence $(R_n)_w=\R(X)_w$
and we conclude by Lemma~\ref{subring},
Lemma~\ref{vample} and
Proposition~\ref{criterion}.
\end{proof}

\begin{remark}
Observe that if we weaken our
assumption on the coefficients
$b_1$ and $b_2$, by allowing
them to have a zero coefficient at
the monomial $T_n^3$, then the GIT
chamber $\lambda([W])$ appearing
in part (2) of Lemma~\ref{vample}
is no longer full-dimensional.
In particular $R_n$ is no longer
the Cox ring $\R(X)$ since otherwise
$\lambda([W])$ would be equal to
$\Nef(X)$, a contradiction.
\end{remark}

\end{document}